\documentclass[a4paper, 11pt]{amsart}

\usepackage{amsfonts}
\usepackage{amssymb,amsthm, amsmath}
\usepackage{xcolor}
\usepackage{url}

\newtheorem{theorem}{Theorem}
\newtheorem{definition}[theorem]{Definition}
\newtheorem{proposition}[theorem]{Proposition}
\newtheorem{corollary}[theorem]{Corollary}
\newtheorem{lemma}[theorem]{Lemma}

\theoremstyle{remark}
\newtheorem{example}[theorem]{Example}

\newtheorem{remark}[theorem]{Remark}
\newtheorem{observation}[theorem]{Observation}

\begin{document}

\title{The short resolution of a semigroup algebra}

\author{I. Ojeda}
\address{Departamento de Matem\'{a}ticas, Universidad de Extremadura,  E-06071 Badajoz, Espa\~na}
\email{ojedamc@unex.es}

\author{{A.} {Vigneron-Tenorio}}
\address{Departamento de Matem\'{a}ticas. Universidad de C\'{a}diz. E-11405 Jerez de la Frontera (C\'{a}diz), Espa\~na}
\email{alberto.vigneron@uca.es}

\thanks{Both authors are partially supported by the projects MTM2012-36917-C03-01 and MTM2015-65764-C3-1-P (MINECO/FEDER, UE), National Plan I+D+I. The first is partially supported by Junta de Extremadura (FEDER funds) and the second by Junta de Andaluc\'{\i}a (group FQM-366)}



\begin{abstract}
This work generalizes the short resolution given in Proc. Amer. Math. Soc. \textbf{131}, 4, (2003), 1081--1091, to any affine semigroup. Moreover, a characterization of Ap\'{e}ry sets is given. This characterization lets compute Ap\'{e}ry sets of affine semigroups and the Frobenius number of a numerical semigroup in a simple way. We also exhibit a new characterization of the Cohen-Macaulay property for simplicial affine semigroups.
\end{abstract}

\maketitle

\section*{Introduction}

Let $\Bbbk$ be a field and let $S$ be a finitely generated commutative submonoid of $\mathbb{Z}^d$ such that $S \cap (-S) = \{0\}$. There is a large literature on the study and computation of minimal free resolutions of the semigroup algebra $\Bbbk[S] = \bigoplus_{\mathbf{a} \in S} \Bbbk \chi^\mathbf{a}$ (see, e.g. \cite{OjVi} and the references therein). Most works on this topic consider $\Bbbk[S]$ as a $\Bbbk[\textbf{X}]( := \Bbbk[X_1, \ldots, X_n])-$module with the structure given by $\Bbbk[\textbf{X}] \to \Bbbk[S];\ X_i \mapsto \chi^{\mathbf{a}_i}$, where $\{\mathbf{a}_1, \ldots, \mathbf{a}_n\}$ is a (fixed) system of generators of $S$.

In \cite{pison03}, Pilar Pis\'on proposed a new and original resolution of $\Bbbk[S]$. She considered $\Bbbk[S]$ as a module over a polynomial ring in fewer variables determined by the extreme rays of the rational cone generated by $S$ and she explicitly constructed a free resolution that she called the \emph{short resolution} of $\Bbbk[S]$. In her construction it is implicitly assumed that the generators of $S$ corresponding to extreme rays are $\mathbb{Z}-$linearly independent. This actually happens when the semigroup is simplicial.

The original aim of this work was to avoid the simplicial hypothesis on the semigroup, by giving the corresponding generalization of Pis\'on's construction in \cite{pison03}. However, during the course of this work, we realized that some improvements can be made so that some results and many proofs from the Pison's paper have been simplified.

One of the original contributions of professor Pis\'on in \cite{pison03} consisted in the explicit computation of test sets for the Apery sets of affine semigroups. To do that she used Gr\"{o}bner bases techniques with respect to a particular local term order. We improve her method by obtaining a new and explicit description of the Apery sets without the use of local term orders (Theorem \ref{Th2}). This allows us to formulate an easy algorithm for the computation of Apery sets of affine semigroups and consequently an algorithm to compute the Frobenius number of a numerical semigroup, as describe in Section \ref{Sect3}. Despite that both algorithms seem to have a good computational behaviour, we would like to emphasize the simplicity of our construction that relies on the computation of just one Gr\"{o}bner basis with respect to a particular (global) term order.

In Section \ref{Sect4}, we give a presentation of any semigroup algebra as a module over a ring in so many variables as the dimension of the cone of the semigroup has (Theorem \ref{Th1}) without assuming the simplicity hypothesis. This completes the construction of the short resolution given in \cite{pison03}. The results of the fourth section, combined with our computational description of the Apery sets, leads to a new characterization of the Cohen-Macaulayness of simplicial affine semigroups (Corollary \ref{Cor CarCM+S}).
  
Finally, in the last section, we propose a new combinatorial description of the Pison's resolution of a semigroup algebra and we explicitly determine the isomorphisms from the combinatorial side to the minimal generators for the presentation given in Section \ref{Sect4}.

\section{Preliminaries}\label{Sect2}

Given a finite subset $\mathcal{A} = \{\mathbf{a}_1, \ldots, \mathbf{a}_n\}$ of $\mathbb Z^d,$ we consider the subsemigroup $S$ of $\mathbb Z^d$ generated by $\mathcal{A},$ that is to say, $S = \mathbb{N} \mathbf{a}_1 + \ldots + \mathbb{N} \mathbf{a}_n,$ where $\mathbb{N}$ denotes the set of nonnegative integers. $S$ is a so-called affine semigroup, in particular, it is finitely generated, cancellative and commutative semigroup with zero element.

Associated to $\mathcal{A}$ is the surjective function $$\deg_\mathcal{A} : \mathbb{N}^n \longrightarrow S;\ \mathbf{u} = (u_1, \ldots, u_n) \longmapsto \deg_\mathcal{A}(\mathbf{u}) = \sum_{i=1}^n u_i \mathbf{a}_i$$ This map is called the factorization map of $S$ in the literature and, accordingly, $\deg^{-1}_\mathcal{A}(\mathbf{a})$ is called the set of factorizations of $\mathbf{a} \in S.$

Notice that the cardinality of $\deg^{-1}_\mathcal{A}(\mathbf{a}),\ \mathbf{a} \in S,$ is not necessarily finite. The necessary and sufficient condition for the finiteness of factorizations is that $S \cap (-S) = 0$ (see \cite[Proposition 1.1]{BCMP}); equivalently, \begin{equation}\label{rc_ecu1} u_1 \mathbf{a}_1 + \ldots + u_n \mathbf{a}_n = 0,\ (u_1, \ldots, u_n) \in \mathbb{N}^n \Longrightarrow u_1 = \ldots = u_n = 0.\end{equation} Throughout this, we will assume that $\mathcal{A}$ satisfies this condition.

Let $\Bbbk$ be a field. The map $\deg_\mathcal{A}$ induces the following surjective $\Bbbk-$algebra homomorphism  $$\begin{array}{rcl} \varphi_\mathcal{A} : \Bbbk[\mathbb{N}^n] = \Bbbk[X_1, \ldots, X_n]& \longrightarrow & \Bbbk[S] := \bigoplus_{\mathbf{a} \in S} \Bbbk\, \chi^\mathbf{a};\smallskip \\ \mathbf{X}^\mathbf{u} := X_1^{u_1} \cdots X_n^{u_n} & \longmapsto & \chi^{\deg_\mathcal{A}(\mathbf{u})}\end{array}$$

Observe that if we consider the grading $\deg(X_i) = \mathbf{a}_i,\ i = 1, \ldots, n,$ on $\Bbbk[X_1, \ldots, X_n],$ then $\varphi_\mathcal{A}$ is homogeneous of degree zero. Hence, both the toric ideal $I_\mathcal{A} := \ker(\varphi_\mathcal{A})$ and the coordinate ring $\Bbbk[S] \cong \Bbbk[X_1, \ldots, X_n]/I_\mathcal{A}$ are homogeneous for the grading determined by $\mathcal{A}$.

In the following, unless otherwise stated, we set $\deg(X_i) = \mathbf{a}_i,\ i = 1, \ldots, n$. That is to say, we will consider $\Bbbk[\mathbf X]$ multigraded by the semigroup $S$.

The necessary and sufficient condition for the finiteness of factorizations assumed above (see formula \ref{rc_ecu1}), implies that there exists a minimal $S-$graded free resolution of $\Bbbk[S]$, which is defined by the property that all of the differentials become zero when tensored with $\Bbbk \cong \Bbbk[\mathbf X]/\mathfrak{m},$ where $\mathfrak{m} = \langle X_1, \ldots, X_m \rangle$ (See \cite[Section 8.3]{MS}). This justify the next definition: 

\begin{definition}
The $i-$th (multigraded) Betti number of $\Bbbk[S]$ in degree $\mathbf{a}$ is $$\beta_{i,\mathbf{a}}(\Bbbk[S]) := \dim_{\Bbbk} \mathrm{Tor}_i^{\Bbbk[\mathbf X]}(\Bbbk, \Bbbk[S])_\mathbf{a}.$$
\end{definition}

\section{Computation of Apery sets}\label{Sect3}

Lt $\mathcal{A} = \{\mathbf{a}_1, \ldots, \mathbf{a}_n\} \subseteq \mathbb Z^d$ be satisfying (\ref{rc_ecu1}), we define the polyhedral cone of $\mathcal{A}$ as follows $$\mathrm{pos}(\mathcal{A}) := \{ \lambda_1 \mathbf{a}_1 + \ldots + \lambda_n \mathbf{a}_n\ \mid\ \lambda_1, \ldots, \lambda_n \in \mathbb{Q}_{\geq 0}\} \subset \mathbb{Q}^d.$$

Without loss of generality, by relabelling if necessary, we may assume that  $\mathrm{pos}(\mathcal{A}) = \mathrm{pos}(\{\mathbf{a}_1, \ldots, \mathbf{a}_r\}), r \leq n.$ Thus, in the following we will write $E = \{\mathbf{a}_1, \ldots, \mathbf{a}_r\}$ and $\mathbf{b}_i = \mathbf{a}_{r+i},\ i = 1, \ldots, s:=n-r.$ So that, \begin{equation}\label{ecu3} \mathcal{A} = E \cup B\end{equation} where $B := \{\mathbf{b}_1, \ldots, \mathbf{b}_s\}.$ This is called a convex partition in \cite[Definition 4.1]{Campillo}. 

\begin{observation}\label{Obs Simplicial}
If $S$ is a simplicial semigroup, that is to say, if $\mathrm{pos}(\mathcal{A})$ can be generated by $\dim_{\mathbb{Q}}(pos(\mathcal{A}))$ 
elements of $\mathcal{A}$, we may take $r$ above equals to $\mathrm{rank}(\mathbb{Z} \mathcal A) = \dim_{\mathbb{Q}}(pos(\mathcal{A}))$, where $
\mathbb{Z} \mathcal A$ denotes the subgroup of $\mathbb{Z}^d$ generated by $\mathcal{A}$. In this case, \eqref{ecu3} is also called a simplicial partition.
\end{observation}

Let $S$ be the semigroup generated by $\mathcal{A}$ and let $\Bbbk[\mathbf{Y}]$ and $\Bbbk[\mathbf{Y},\mathbf{Z}]$ be the polynomial rings in $r$ and $n$ variables, respectively, over a field $\Bbbk.$

Let $\prec$ be a monomial order on $\Bbbk[\mathbf{Y},\mathbf{Z}]$ defined as follows: $\mathbf{Y}^{\mathbf{v}'} \mathbf{Z}^{\mathbf{u}'} \prec \mathbf{Y}^\mathbf{v} \mathbf{Z}^\mathbf{u}$ if and only if the leftmost nonzero entry of $\deg_\mathcal{A}(\mathbf{u},\mathbf{v}) - \deg_\mathcal{A}(\mathbf{u}',\mathbf{v}')$ is positive or $\deg_\mathcal{A}(\mathbf{u},\mathbf{v}) = \deg_\mathcal{A}(\mathbf{u}',\mathbf{v}')$ and $\mathbf{Y}^{\mathbf{v}'} \mathbf{Z}^{\mathbf{u}'} \prec_{revlex} \mathbf{Y}^\mathbf{v} \mathbf{Z}^\mathbf{u}$, where $\prec_{revlex}$ is a reverse lexicographic ordering on $\Bbbk[\mathbf{Y},\mathbf{Z}]$ such that $Y_i \prec Z_j,$ for every $i = 1, \ldots, r$ and $j = 1, \ldots, s.$  By abusing of terminology, we  will say that $\prec$ is an $\mathcal{A}-$graded reverse lexicographical monomial order on $\Bbbk[\mathbf{Y},\mathbf{Z}]$ such that $Y_i \prec Z_j,$ for every $i = 1, \ldots, r$ and $j = 1, \ldots, s.$

Set $\varphi_\mathcal{A} : \Bbbk[\mathbf{Y},\mathbf{Z}] \to  \Bbbk[S];\ \mathbf{Y}^\mathbf{v} \mathbf{Z}^\mathbf{u} \mapsto \chi^{\deg_\mathcal{A}(\mathbf{v},\mathbf{u})}$ and let $\mathcal{G}_\prec(I_\mathcal{A})$ be the reduced Gr\"obner basis of $I_\mathcal{A} = \ker(\varphi_\mathcal{A})$ with respect to $\prec$. We will write $\mathcal{Q}$ for the exponents of the standard monomials in the variables $Z_1, \ldots, Z_s,$ that is to say, $$\mathcal{Q} = \{ \mathbf{u} \in \mathbb{N}^s \mid \mathbf{Z}^\mathbf{u} \not\in \mathrm{in}_\prec(I_\mathcal{A})\}.$$

\begin{proposition}
The set $\mathcal{Q}$ is finite.
\end{proposition}

\begin{proof}
Since $\mathbf{b}_j \in \mathrm{pos}(\mathcal{A}) = \mathrm{pos}(\{\mathbf{a}_1, \ldots, \mathbf{a}_r\})$, there exist $u_j, v_{1j}, \ldots, v_{rj} \in \mathbb{N}$ such that $u_j \mathbf{b}_j = \sum_{i = 1}^r v_{ij} \mathbf{a}_i,$ for every $j = 1, \ldots, s$ (i.e. $Z_j^{u_j} - Y_1^{v_1} \cdots Y_r^{v_r} \in I_\mathcal{A},$ for each $j$). Therefore, $Z^{\mathbf{u}'} \in \mathrm{in}_\prec(I_\mathcal{A}),$ for every $\mathbf{u}' \in \mathbb{N}^s$ whose $j-$th coordinate is larger than $u_j$ for some $j = 1, \ldots, s$.
\end{proof}

We recall that 	the Ap\'ery set of $S$ relative to $E,\ \mathrm{Ap}(S,E)$, is defined as $$\mathrm{Ap}(S,E) = \{ \mathbf{a} \in S\ \mid\ \mathbf{a} - \mathbf{e} \not\in S,\ \forall \mathbf{e} \in E \}.$$

Our main result in this section improves \cite[Lemma 1.2]{pison03}. In contrast to \cite{pison03}, we are considering a global monomial order. This will have important consequences for the forthcoming constructions.

Observe that the natural injection $\iota : \mathbb{N}^s \hookrightarrow \mathbb{N}^n; \mathbf{u} \mapsto (0,\mathbf{u})$ allows us to restrict $\deg_\mathcal{A} \circ\ \iota(-)$ to $\mathcal{Q}$.

\begin{theorem}\label{Th2}
The restriction of $\deg_\mathcal{A} \circ\ \iota (-)$ to $\mathcal{Q}$ defines a bijective map $\mathcal{Q} \to \mathrm{Ap}(S,E)$.
\end{theorem}

\begin{proof}
Let $\mathbf{u} \in \mathcal{Q}$ and set $\mathbf{q} = \deg_\mathcal{A}\big(\iota(\mathbf{u})\big).$ If $\mathbf{q} \in \mathrm{Ap}(S,E),$ then there exists $i \in \{1, \ldots, r\}$ such that  $\mathbf{q}-\mathbf{a}_i = \sum_{i=1}^r v_i \mathbf{a}_i + \sum_{j=1}^s w_j \mathbf{b}_j \in S$, thus we have a binomial $\mathbf{Z}^\mathbf{u} - Y_i \mathbf{Y}^\mathbf{v} \mathbf{Z}^\mathbf{w} \in I_\mathcal{A}$. Since the monomials $\mathbf{Z}^\mathbf{u}$ and $Y_i \mathbf{Y}^\mathbf{v} \mathbf{Z}^\mathbf{w}$ are distinct and $\mathbf{Z}^\mathbf{u} \prec Y_i \mathbf{Y}^\mathbf{v} \mathbf{Z}^\mathbf{w}$ because $Y_i$ divides $Y_i \mathbf{Y}^\mathbf{v} \mathbf{Z}^\mathbf{w}$, we conclude that $\mathbf u \notin \mathcal Q$, a contradiction.  Therefore, the image of the restriction of $\deg_\mathcal{A} \circ\ \iota(-)$ to $\mathcal{Q}$ lies in $\mathrm{Ap}(S,E)$.

Consider now $\mathbf{q} \in \mathrm{Ap}(S,E).$ Then, $\mathbf{q}$ admits a factorization in the form $\mathbf{q} = \sum_{i=1}^s v_i \mathbf{b}_i.$ The remainder of the division of $\mathbf{Z}^\mathbf{v}, \mathbf{v} = (v_1, \ldots, v_s) \in \mathbb{N}^s,$ by $\mathcal{G}_\prec(I_\mathcal{A})$ is a monomial $\mathbf{Z}^\mathbf{u}$ of $\mathcal{A}-$degree $\mathbf{q}$ which does not lie in $\mathrm{in}_\prec(I_\mathcal{A}).$ Hence $\mathbf{u} \in \mathcal{Q}$ and $\deg_\mathcal{A}\big(\iota(\mathbf{u})\big) = \mathbf{q}$ which proves the surjectivity of our map.

Finally, in order to prove that $\deg_\mathcal{A}\big(\iota(\mathbf{u})\big) = \deg_\mathcal{A}\big(\iota(\mathbf{v})\big)$ implies $\mathbf{u} = \mathbf{v},$ it suffices to observe that $f := \mathbf{Z}^\mathbf{u} - \mathbf{Z}^\mathbf{v} \in I_\mathcal{A}.$ So, if $\mathbf{u} \neq \mathbf{v},$ then $\mathrm{in}_\prec(f) = \mathbf{Z}^\mathbf{u}$ (or $\mathrm{in}_\prec(f) = \mathbf{Z}^\mathbf{v}$), that is to say, $\mathbf{u} \not\in \mathcal{Q}$ (or $\mathbf{v} \not\in \mathcal{Q}$) which leads us to a contradiction.
\end{proof}

Notice that Theorem \ref{Th2} gives an easy algorithm for the computation of Ap\'ery sets. A similar algorithm, based on a purely semigroup approach, can be found in \cite{MOT}. 

In the particular case when $S$ is a numerical semigroup (that is, to say if $S$ is a submonoid of $\mathbb{N}$ with finite complement in $\mathbb{N}$), Theorem \ref{Th2} gives an algorithm for computing the Frobenius number of $S, g(S)$ (i.e., the greatest natural number not belonging to $S$). It suffices to recall the well-known formula due to R. Ap\'{e}ry (see, e.g., \cite[Proposition 10.4]{GSR})
$$
g(S) = \max\{\mathrm{Ap}(S,\mathbf{a}_1)\} - \mathbf{a}_1
$$
and the fact that any nonzero element of a numerical semigroup generates the corresponding polyhedral cone in $\mathbb{Q}$. 

It is fair to point out that Marcel Morales and Nguyen Thi Dung have recently produced an algorithm by using similar arguments for the computation of the Frobenius number (see \cite{Morales}). Professor Morales has informed us that similar techniques were used by Einstein et al. \cite{einstein} and Roune \cite{roune} to give sophisticated algorithms for the computation of the Frobenius number of numerical semigroup. However, in these papers the important role of the Apery sets is not observed.

\begin{example}
The following example is taken from \cite{Zarzuela}. Let $\mathcal{A} = \{8,11,18\}$ and let $\prec$ be the
$\mathcal{A}-$graded reverse lexicographical monomial order on $\Bbbk[Y,Z_1,Z_2],$ such that $Y \prec Z_2 \prec Z_1$. We computed with Singular \cite{DGPS} the reduced Gr\"{o}bner basis of $I_\mathcal{A} \subseteq \Bbbk[Y,Z_1,Z_2]$ with respect $\prec$: $$\mathcal{G}_\prec(I_\mathcal{A}) = \{Z_1^2 Z_2-Y^5, Z_1^4-Z_2^2 Y, Z_2^3-Z_1^2 Y^4\}.$$ Clearly, $\mathcal{Q} = \{1, Z_1, Z_2, Z_1^2, Z_1 Z_2, Z_2^2, Z_1^3, Z_1 Z_2^2 \}$ and $$\mathrm{Ap}(S,\{8\}) = \deg_\mathcal{A}(\mathcal{Q}) = \{0,11,18,22,29,36,33,47\}.$$
In this case, the Frobenius number is $47-8 = 39$.

The whole process can be automated easily, as the following Singular code shows:

\begin{verbatim}
     LIB "toric.lib";
     LIB "general.lib";
     intmat A[1][3] = 18,11,8;
     ring r = 0, (Z(1..size(A)-1), Y), dp;
     ideal i = toric_ideal(A,"hs");
     ring s = 0, (Z(1..size(A)-1), Y), wp(A);
     ideal i = imap(r,i);
     ideal m = lead(std(i));
     ideal Q = kbase(std(m+Y));
     int n;
     intmat Ap[1][size(Q)];
     for (n = 1; n <= size(Q); n = n + 1)
         {Ap[1,n] = A*intmat(leadexp(Q[n]));}
     int g = sort(intvec(Ap))[1][size(Q)]-A[1,size(A)];
     g;
\end{verbatim}
\end{example}

The interested reader is encouraged to change the third line of the code above, in order to compute the Frobenius number of his or her favorite semigroup.

The first author joint with {C.J.}~Moreno have written a function in Singular (\cite{DGPS}) for the computation of the Ap\'ery set and the Frobenius number of a numerical semigroup. The library is available at \url{http://matematicas.unex.es/~ojedamc/inves/apery.lib}. Using this library, we have computed the Frobenius of the numerical semigroup in \cite[Example 5.5]{Morales} in less than $0.6$ seconds with an Intel$\text{}^\copyright$ Core$\text{}^{\text{TM}}$ i5-2450M CPU @ 2.50GHz$\times$4.

At Marcel Morales' webpage, \url{https://www-fourier.ujf-grenoble.fr/~morales/}, it can be found a software called {\tt Frobenius-public.exe} for computing the Ap\'{e}ry set and the Frobenius number of a numerical semigroup. This software uses the algorithms presented in \cite{Morales}. We have used this software and the previous Singular's library {\tt apery.lib} to compare the computational behaviour of our algorithms and the algorithms in \cite{Morales}. In general, these algorithms show a similar behaviour, but for numerical semigroups with large Frobenius number our algorithm could be a little better. For example, the Frobenius number of the semigroup generated by \begin{align*} \{ 1051, 1071, 1087, 1099  ,1129, 1139,1199, 1207,\\ 1211, 1213, 3331,4325, 5511, 10311, 11421\}, \end{align*} is $11703$. The program {\tt apery.lib} needs one second to compute its Ap\'{e}ry set and its Frobenius number, and {\tt Frobenius-public.exe} needs approximately fifteen seconds.

\section{Pison's free resolution}\label{Sect4}

We keep the notation of the previous section.
Let $S_E$ be the subsemigroup of $S$ generated by $E$ and set $$\Bbbk[S_E] := \bigoplus_{\mathbf{a} \in S_E} \Bbbk \chi^\mathbf{a}.$$ The composition $\Bbbk[\mathbf{Y}] \stackrel{\varphi_E}{\longrightarrow} \Bbbk[S_E] \hookrightarrow \Bbbk[S]$ defines a natural structure of $\Bbbk[\mathbf{Y}]-$module on $\Bbbk[S].$

Obviously, $\Bbbk[\mathbf{Y}]$ is multigraded by $S$. So, there exists a minimal $\mathcal{A}-$graded free resolution of $\Bbbk[S]$ as $\Bbbk[\mathbf{Y}]-$module (see \cite[Section 8.3]{MS}). In order to compute effectively this resolution, an $\mathcal{A}-$graded presentation of $\Bbbk[S]$ as $\Bbbk[\mathbf{Y}]-$module is required.

\begin{proposition}
The set $\{ \mathbf{Z}^\mathbf{u}\ \mid\ \mathbf{u} \in \mathcal{Q} \}$ is a minimal system of generators of $\Bbbk[S]$ as $\Bbbk[\mathbf{Y}]-$module.
\end{proposition}

\begin{proof}
Since $\Bbbk[S] \cong \Bbbk[\mathbf{Y}, \mathbf{Z}] / I_\mathcal{A},$ the result follows from the definition of $\mathcal{Q}.$
\end{proof}

\begin{remark}\label{Remark6}
Observe that $S$ is a simplicial semigroup if and only if $I_{S_E} = I_S \cap \Bbbk[\mathbf{Y}] = 0$. In this case, $\varphi_E$ is an isomorphism. This condition on $S$ is implicitly assumed in \cite[Section 1]{pison03}.
\end{remark}

In order to give an $\mathcal{A}-$graded presentation of $\Bbbk[S]$ as $\Bbbk[\mathbf{Y}]-$module in the general setting. We first order $\mathcal{Q}$ lexicographically; so that there is a bijection $\sigma$ from $\{1, \ldots, \beta_0 := \#\mathcal{Q} \}$ to $\mathcal{Q}$. Now, we may define the following surjective $\Bbbk[\mathbf{Y}]-$module homomorphism
$$\psi_0 : \Bbbk[\mathbf{Y}]^{\beta_0} \longrightarrow \Bbbk[S]$$ with $\psi_0(\varepsilon_i) = \mathbf{Z}^{\sigma(i)},\ i = 1, \ldots, \beta_0,$ where $\{\varepsilon_1, \ldots, \varepsilon_{\beta_0}\}$ is the canonical basis of $\Bbbk[\mathbf{Y}]^{\beta_0}$.

Let $\mathbf{Y}^\mathbf{v} \mathbf{Z}^\mathbf{u} - \mathbf{Y}^{\mathbf{v}'} \mathbf{Z}^{\mathbf{u}'}$ be an element of $\mathcal{G}_\prec(I_\mathcal{A})$ whose leading term is $\mathbf{Y}^\mathbf{v} \mathbf{Z}^\mathbf{u}$ with $\mathbf{v} \neq 0$ and $\mathbf{u} \neq 0$. First of all, we notice that $\mathbf{v}' \neq 0$ which implies $\mathbf{Z}^\mathbf{u}$ and $\mathbf{Z}^{\mathbf{u}'} \in \mathcal{Q}$. Moreover, since no variable is zero divisor modulo $I_\mathcal{A}$ (because $I_\mathcal{A}$ is a toric ideal), we have that $\mathbf{u} \neq \mathbf{u}'$, and we conclude that $\mathbf{Y}^\mathbf{v} - \mathbf{Y}^{\mathbf{v}'} \in I_\mathcal{A}$, in contradiction with $\mathcal{G}_\prec(I_\mathcal{A})$ to be reduced. Now, for each $\mathbf{w} \in \mathbb{N}^s$ such that $\mathbf{Z}^{\mathbf{u} + \mathbf{w}} \in \mathcal{Q}$, consider the remainder, $\mathbf{Y}^{\mathbf{w}'} \mathbf{Z}^{\mathbf{u}''}$, of $\mathbf{Z}^{\mathbf{u}' + \mathbf{w}}$ on division by $\mathcal{G}_\prec(I_\mathcal{A})$ (which may be $\mathbf{Z}^{\mathbf{u}' + \mathbf{w}}$ itself) and define the element $\mathbf{f} \in \Bbbk[\mathbf{Y}]^{\beta_0}$ whose $\sigma^{-1}(\mathbf{u} + \mathbf{w})-$th and $\sigma^{-1}(\mathbf{u}'')-$th coordinates are $\mathbf{Y}^\mathbf{v}$ and $-\mathbf{Y}^{\mathbf{v}'+\mathbf{w}'}$, respectively, and zeroes elsewhere. Observe that $\psi_0(\mathbf{f}) = 0$

Let \begin{equation}\label{ecu2} \mathcal{M}' = \{\textbf{f}_1, \ldots, \textbf{f}_{\beta'_0}\} \subset \Bbbk[\mathbf{Y}]^{\beta_0}\end{equation} be the set of elements of $\Bbbk[\mathbf{Y}]^{\beta_0}$ defined as above and let $M'$ be the $\beta_0 \times \beta'_0-$matrix whose columns are $\textbf{f}_1, \ldots, \textbf{f}_{\beta'_0}.$

If $I_{S_E} \neq 0,$ then there exists  $\mathbf{Y}^\mathbf{v} \mathbf{Z}^\mathbf{u} - \mathbf{Y}^{\mathbf{v}'} \mathbf{Z}^{\mathbf{u}'} \in \mathcal{G}_\prec(I_\mathcal{A})$ whose leading term is $\mathbf{Y}^\mathbf{v} \mathbf{Z}^\mathbf{u}$ with $\mathbf{v} \neq 0$ and $\mathbf{u} = 0$; in particular, $\mathbf{u}' = 0$; otherwise the leading term would be
$\mathbf{Y}^{\mathbf{v}'} \mathbf{Z}^{\mathbf{u}'}$. Let $\{g_1, \ldots, g_t\} \subset \Bbbk[\mathbf{Y}]$ be a (minimal) system of binomial generators of $I_{S_E}$ and define the $\Bbbk[\mathbf{Y}]-$module generated by the columns of the matrix $$N = \mathbf{1}_{\beta_0} \otimes (g_1\ \ldots\ g_t),$$ where $\mathbf{1}_{\beta_0}$ denotes the identity matrix of size $\beta_0$ and $\otimes$ denotes the Kronecker product of matrices.

Set $\beta_1 = \beta_1' + t \cdot \beta_0$. Clearly, $M := (M' \vert N)$ defines a homomorphism of free $\Bbbk[\mathbf{Y}]-$modules $\psi_1 :\Bbbk[\mathbf{Y}]^{\beta_1} \to \Bbbk[\mathbf{Y}]^{\beta_0}$ such that $\mathrm{im}(\psi_1) \subseteq \ker(\psi_0)$. If $I_{S_E} = 0,$ we take $t=0$ and $M = M'$ (this is the case in \cite[Section 1]{pison03}).

\begin{theorem}\label{Th1}
With the notation above, $\mathrm{im}(\psi_1) = \ker(\psi_0)$. In particular, $\mathrm{coker}(\psi_1) \cong_{\Bbbk[\mathbf{Y}]} \Bbbk[S]$, that is to say, $\psi_1$ is a presentation of $\Bbbk[S]$ as $\Bbbk[\mathbf{Y}]-$module.
\end{theorem}

\begin{proof}
By construction, it suffices to prove that $\mathrm{im}(\psi_1) \supseteq \ker(\psi_0)$.

Let $f_1, \ldots, f_{\beta_0} \in \Bbbk[\mathbf{Y}]$ be such that $\mathbf{f} = (f_1, \ldots, f_{\beta_0})^\top \in \mathrm{ker}(\psi_0)$, where $\top$ denotes transpose. By hypothesis, $f = \sum_{i=1}^{\beta_0} f_i \mathbf{Z}^{\sigma(i)} \in I_\mathcal{A}$. Without loss of generality, we may suppose that $f_i \mathbf{Z}^{\sigma(i)}$ is homogeneous of $\mathcal{A}-$degree $\mathbf{a}$, for every $i = 1, \ldots, \beta_0$.

If $f \neq 0,$ then its leading term is $\mathbf{Y}^{\mathbf{v}'} \mathbf{Z}^{\mathbf{u}'}$ with $\mathbf{v}' \neq 0$ and $\mathbf{u}' = \sigma(i)$ for some $i.$ Let $g \in \mathcal{G}_\prec(I_\mathcal{A})$ be an element whose leading term, $\mathbf{Y}^\mathbf{v} \mathbf{Z}^\mathbf{u}$, divides
$\mathbf{Y}^{\mathbf{v}'} \mathbf{Z}^{\mathbf{u}'}$.

If $\mathbf{u} \neq 0,$ let $\mathbf{w} = \mathbf{u}' - \mathbf{u} $ and consider the element $\mathbf{f}_j \in \mathcal{M}$ corresponding to $g$ and $\mathbf{w}$. In this case, we obtain that  $\mathbf{f} - \psi_1(\varepsilon_j) := (f'_1, \ldots, f'_{\beta_0})^\top \in \ker(\psi_0)$, where $\varepsilon_j$ is the $j-$th vector of the canonical basis of $\Bbbk[\mathbf{Y}]^{\beta_1}$, and the leading term of $f' = \sum_{j=1}^{\beta_0} f'_j \mathbf{Z}^{\sigma(i)}$ is lesser than the leading term of $f$.

On the other hand, if $\mathbf{u} = 0,$ then $g = \mathbf{Y}^\mathbf{v} - \mathbf{Y}^{\mathbf{v}''} \in I_{S_E}$. Therefore, $g = \sum_{j=1}^t h_j g_j$. Let $H$ be the $t \times \beta_0-$matrix whose $i-$th column is $(h_1 \ldots h_t)^\top$ and define $\mathbf{h}_i = \binom{\mathbf{0}}{\mathrm{vec}(H)} \in \Bbbk[\mathbf{Y}]^{\beta_1},$ where $\mathrm{vec}(-)$ denotes the vectorization operator and $\mathbf{0}$ is a vector of zeroes. Clearly, $\mathbf{f} - \psi_1(\mathbf{h}_i) = (f'_1, \ldots, f'_{\beta_0})^\top \in \ker(\psi_1)$ and the leading term of $f' = \sum_{j=1}^{\beta_0} f'_j \mathbf{Z}^{\sigma(i)}$ is lesser than the leading term of $f$.

Now, we repeat the same process to $(f'_1, \ldots, f'_{\beta_0})^\top$, and so on. Since in each step the leading term of the corresponding polynomial in $\Bbbk[\mathbf{Y}, \mathbf{Z}]$ decreases, this process must terminate.
\end{proof}

Recall that an $\mathcal{A}-$graded free resolution of $\mathrm{coker}(\varphi_1)$ as $\Bbbk[\mathbf{Y}]-$module is an acyclic complex of length $t \leq r$ $$\mathcal{P} : \Bbbk[\mathbf{Y}]^{\beta_t} \stackrel{\psi_t}{\longrightarrow} \ldots \longrightarrow \Bbbk[\mathbf{Y}]^{\beta_1} \stackrel{\psi_1}{\longrightarrow} \Bbbk[\mathbf{Y}]^{\beta_0} \longrightarrow \mathrm{coker}(\psi_1), $$ where the maps are all homogeneous of $\mathcal{A}-$degree $0$. Since, by Theorem \ref{Th1}, $\mathrm{coker}(\psi_1) \cong_{\Bbbk[\mathbf{Y}]} \Bbbk[S]$, we call $\mathcal{P}$ a \textbf{Pis\'on's free resolution} of $\Bbbk[S].$ Notice that both the isomorphism and $\psi_1$ are given explicitly, so this resolution can be effectively computed.

\begin{corollary}
With the notation above, if the subgroup of $\mathbb{Z}^d$ generated by $B,\ \mathbb Z B$, is contained in $S_E \cup (-S_E)$, then $i-$th map in the Pison's free resolution of $\Bbbk[S]$ can be taken to be the direct sum of $\#\mathcal{Q}$ copies of the $i-$th map in a minimal free resolution of $\Bbbk[S_E]$, for every $i > 0$.
\end{corollary}

\begin{proof}
We claim that the set $\mathcal{M}'$ defined in (\ref{ecu2}) is empty. Otherwise, there exists $\mathbf{Y}^\mathbf{v} \mathbf{Z}^\mathbf{u} - \mathbf{Y}^{\mathbf{v}'} \mathbf{Z}^{\mathbf{u}'} \in \mathcal{G}_\prec(I_\mathcal{A})$ whose leading term is $\mathbf{Y}^\mathbf{v} \mathbf{Z}^\mathbf{u}$ with $\mathbf{v} \neq 0$ and $\mathbf{u} \neq 0$. Since $\sum_{i=1}^s (u_i - u'_i) \mathbf{b}_i \in \mathbb Z B$, by hypothesis, there exist $\mathbf{w}_i \in \mathbb{N},\ i = 1, \ldots, r$, such that $\sum_{i=1}^s (u_i - u'_i) \mathbf{b}_i = \pm \sum_{i=1}^r w_i \mathbf{e}_i$. Therefore, either $\sum_{i=1}^s u_i \mathbf{b}_i + \sum_{i=1}^r w_i \mathbf{e}_i = \sum_{i=1}^s u'_i  \mathbf{b}_i$ or  $\sum_{i=1}^s u_i \mathbf{b}_i  = \sum_{i=1}^s u'_i  \mathbf{b}_i + \sum_{i=1}^r w_i \mathbf{e}_i$, that is to say, either $\mathbf{Z}^{\mathbf{u}'} - \mathbf{Y}^\mathbf{w} \mathbf{Z}^{\mathbf{u}} \in I_\mathcal{A}$ or  $\mathbf{Z}^{\mathbf{u}} - \mathbf{Y}^\mathbf{w} \mathbf{Z}^{\mathbf{u}'} \in I_\mathcal{A}$, in contradiction with the reducedness of $\mathcal{G}_\prec(I_\mathcal{A})$.
\end{proof}

\begin{example}
Let $$A = \left(\begin{array}{ccccccc}
3 & 1 & 1 & 1 & 2 & 4 & 1\\
1 & 3 & 1 & 1 & 0 & 0 & 1\\
1 & 1 & 3 & 1 & 4 & 2 & 2\\
1 & 1 & 1 & 3 & 0 & 0 & 2
\end{array}\right)$$ and consider the subsemigroup $S$ of $\mathbb{N}^4$ generated by the columns, $\mathbf{a}_1, \ldots, \mathbf{a}_6$ and $\mathbf{b}$, of $A$. Set $\mathcal{A} = \{\mathbf{a}_1, \ldots, \mathbf{a}_6, \mathbf{b}\}$ and $E = \{\mathbf{a}_1, \ldots, \mathbf{a}_6\}$, clearly $\mathrm{pos}(\mathcal{A}) = \mathrm{pos}(E)$ and $\mathbb Z S \subset S_E \cup (-S_E)$. The ideal  $I_\mathcal{A} \subseteq \Bbbk[Y_1, \ldots, Y_6, Z]$ is equal to $\langle Z^2 - Y_3 Y_4 \rangle + I_{S_E}$. Therefore, $\mathcal{Q} = \{0,1\}$, $$M = \left(\begin{array}{cccccccc} g_1 & g_2 & g_3 & g_4 & 0 & 0 & 0 & 0 \\ 0 & 0 & 0 & 0 & g_1 & g_2 & g_3 & g_4 \end{array}\right),$$ where $\{g_1 = Y_1Y_5-Y_3Y_6, g_2 = Y_1Y_3^3-Y_2Y_4Y_5^2, g_3 = Y_1^2Y_3^2-Y_2Y_4Y_5Y_6, g_4 = Y_1^3Y_3-Y_2Y_4Y_6^2\}$ is a minimal system of generators of $I_{S_E}$, and we conclude that the Pis\'on's free resolution of $\Bbbk[S]$ is
$$0 \to \Bbbk[\mathbf{Y}]^2 \stackrel{\phi_3 \oplus \phi_3}{\longrightarrow} \Bbbk[\mathbf{Y}]^8 \stackrel{\phi_2 \oplus \phi_2}{\longrightarrow} \Bbbk[\mathbf{Y}]^8 \stackrel{\phi_1 \oplus \phi_1}{\longrightarrow} \Bbbk[\mathbf{Y}]^2 \stackrel{\psi_0}{\longrightarrow} \Bbbk[S],$$ where $$0 \to \Bbbk[\mathbf{Y}] \stackrel{\phi_3}{\longrightarrow} \Bbbk[\mathbf{Y}]^4 \stackrel{\phi_2}{\longrightarrow} \Bbbk[\mathbf{Y}]^4 \stackrel{\phi_1}{\longrightarrow} \Bbbk[\mathbf{Y}] \stackrel{\varphi_E}{\longrightarrow} \Bbbk[S_E]$$ is a minimal free resolution of $\Bbbk[S_E]$.
\end{example}

\begin{lemma}\label{BH1.2.26}
With the notation above, $\mathrm{depth}_{\Bbbk[\mathbf Y, \mathbf Z]}(\Bbbk[S]) = \mathrm{depth}_{\Bbbk[\mathbf Y]}(\Bbbk[S])$
\end{lemma}

\begin{proof}
Since we are assuming that $S \cap (-S) = 0$, both $\Bbbk[\mathbf Y, \mathbf Z]$ and $\Bbbk[\mathbf Y]$ can be regarded as local rings with maximal ideals $\langle Y_1, \ldots, Y_r, Z_1, \ldots, Z_s \rangle$ and $\langle Y_1, \ldots, Y_r \rangle$, respectively, because of the graduation given by the semigroup $S$. Clearing, the natural projection $\Bbbk[\mathbf Y, \mathbf Z] \to \Bbbk[\mathbf Y]$ is an homomorphism of local rings. So, our claim follows from \cite[Exercise 1.2.26(b)]{BH}.
\end{proof}

\begin{corollary}\label{Cor CarCM+S}
With the notation above, if $S$ is a simplicial semigroup, then $\Bbbk[S]$ is Cohen-Macaulay if and only if the generators of $\mathrm{in}_\prec(I_S)$ do not depend on $Y_1, \ldots, Y_r$.
\end{corollary}

\begin{proof}
Since $S$ is simplicial, we may assume that $\dim(\mathrm{pos}(\mathcal{A})) = r$ (see Observation \ref{Obs Simplicial}); so, the Krull dimension of $\Bbbk[S]$ equals $r$ (see the proof of \cite[Proposition 7.5]{MS}). Therefore, $\Bbbk[S]$ is Cohen-Maculay if and only if $\mathrm{depth}_{\Bbbk[\mathbf Y, \mathbf Z]}(\Bbbk[S]) = r$. Now, since $\mathrm{depth}_{\Bbbk[\mathbf Y, \mathbf Z]}(\Bbbk[S]) = \mathrm{depth}_{\Bbbk[\mathbf Y]}(\Bbbk[S])$ by Lemma \ref{BH1.2.26} and $\Bbbk[\mathbf Y] = \Bbbk[S_E]$ because $I_{S_E} = 0$ (see Remark \ref{Remark6}), from the Auslander-Buchbaum formula it follows that $\Bbbk[S]$ is Cohen-Maculay if and only if the projective dimension of $\Bbbk[S]$ as $\Bbbk[\mathbf Y]-$module is $0$. Equivalently, $$\psi_0 : \Bbbk[\mathbf Y]^{\# \mathcal{Q}} \cong_{\Bbbk[\mathbf Y]}\Bbbk[S]$$ which means that $\mathrm{in}_\prec(I_S)$ is minimally generated in $\Bbbk[\mathbf{Z}]$, as it is deduced from our construction.
\end{proof}

\begin{example}
Let $$\mathcal{A} = \big\{(6,1),(6,2),(6,3),(7,2),(7,3),(8,2),(8,3),(9,3), (10,3) \big\} \subset \mathbb{Z}^2$$ and $\Bbbk[\mathbf Y, \mathbf Z] = \Bbbk[Y_1, Z_1, Y_2, Z_2, \ldots, Z_7 ]$. Let $\prec$ be the $\mathcal{A}-$graded reverse lexicographical term ordering on $\Bbbk[\mathbf Y, \mathbf Z]$ such that $Y_1 \prec Y_2 \prec Z_1 \prec \ldots \prec Z_7$. The computation of the minimal system of generators of $\mathrm{in}_\prec(I_\mathcal{A})$ can be done with Singular \cite{DGPS}:

\begin{verbatim}
     LIB "toric.lib";
     option(redSB);
     intmat A[2][9] = 6,6,6,7,7,8,8,9,10,
                      1,3,2,2,3,2,3,3,3;
     intmat B[9][9] = 6,6,6,7,7,8,8,9,10,
                      1,3,2,2,3,2,3,3,3,
                     -1,0,0,0,0,0,0,0,0,
                      0,-1,0,0,0,0,0,0,0,
                      0,0,0,0,0,0,0,0,1,
                      0,0,0,0,0,0,0,1,0,
                      0,0,0,0,0,0,1,0,0,
                      0,0,0,0,0,1,0,0,0,
                      0,0,0,0,1,0,0,0,0;
     ring r = 0, (Y1,Y2,Z1,Z2,Z3,Z4,Z5,Z6,Z7), dp;
     ideal i = toric_ideal(A,"hs");
     ring s = 0, (Y1,Y2,Z1,Z2,Z3,Z4,Z5,Z6,Z7), M(B);
     ideal i = imap(r,i);
     i = groebner(i);
     ideal m = lead(i);
\end{verbatim}
Now, since $\mathrm{in}_\prec(I_\mathcal{A}) = \big\langle Z_1^2, Z_1Z_2, Z_1Z_3, Z_2^2, Z_1Z_4, Z_2Z_3, Z_1Z_5, Z_3^2, Z_2Z_4, Z_2Z_5,$ $Z_1Z_6, Z_3Z_5,$ $Z_4^2,$ $Z_2Z_6,$ $Z_1Z_7, Z_5^2, Z_3Z_6, Z_2Z_7, Z_3Z_7, Z_4Z_7, Z_6^2, Z_5Z_7, Z_6Z_7, Z_7^2,$ $Z_4Z_5Z_6 \big\rangle,$ by Corollary \ref{Cor CarCM+S}, we conclude that the semigroup algebra of the subsemigroup of $\mathbb{Z}^2$ generated by $\mathcal{A}$ is Cohen-Macaulay.
\end{example}

As a consequence of Corollary \ref{Cor CarCM+S} we obtain a formula for the Castelnouvo-Mumford regularity of $I_S$ in terms of the set $\mathcal{Q}$ when $S$ a simplicial semigroup and $\Bbbk[S]$ is Cohen-Macaulay.

\begin{corollary}
With the notation above, if $S$ a simplicial semigroup, $\Bbbk[S]$ is Cohen-Macaulay such and $I_S$ is homogeneous for the standard grading, then the Castelnouvo-Mumford regularity of $I_S$ is $$\mathrm{reg}(I_S) = \max \Big\{ \sum_{i=1}^r u_i\ \mid\ \deg_\mathcal{A}(\mathbf{u}) \in \mathcal{Q} \Big\}.$$
\end{corollary}

\begin{proof}
By the proof of Corollary \ref{Cor CarCM+S}, $\Bbbk[\mathbf Y]^{\# \mathcal{Q}} \cong_{\Bbbk[\mathbf Y]}\Bbbk[S]$. Now, this is a particular case of \cite[Theorem 16]{CBPV}.
\end{proof}

\section{Combinatorial description}

We end this paper by giving a new combinatorial description of the Pison's resolution. Again, we keep the notation of the previous sections.

Let $$\mathcal{P} : \Bbbk[\mathbf{Y}]^{\beta_t} \stackrel{\psi_t}{\longrightarrow} \ldots \longrightarrow \Bbbk[\mathbf{Y}]^{\beta_1} \stackrel{\psi_1}{\longrightarrow} \Bbbk[\mathbf{Y}]^{\beta_0} \stackrel{\psi_0}{\longrightarrow}  \Bbbk[S],$$ be an $\mathcal{A}-$graded free resolution of $\Bbbk[S]$ as $\Bbbk[\mathbf{Y}]-$module, that is to say, a Pison's free resolution of $\Bbbk[S]$. Set $M_i = \ker(\psi_i),\ i = 0, \ldots, t,\ \mathfrak{m}_E$ equals to the irrelevant ideal of $\Bbbk[\mathbf{Y}]$ and $W_i(\mathbf{a}) = (M_i/\mathfrak{m}_E M_i)_\mathbf{a},$ with $\mathbf{a} \in S$. Since $W_i(\mathbf{a}) \cong \mathrm{Tor}_{i}^{\Bbbk[\mathbf Y]}(\Bbbk, \Bbbk[S])_\mathbf{a},$ the $i-$th Betti number of $\Bbbk[S]$ of degree $\mathbf{a}$ is $\dim_\Bbbk(W_i(\mathbf{a}))$. Thus,
$$\beta_{i} = \sum_{\mathbf{a} \in S} \dim_\Bbbk(W_i(\mathbf{a})),$$ for each $i = 0, \ldots, t.$

The following abstract simplicial complexes $$T_\mathbf{a} = \{ F \subseteq E \mid \mathbf{a} - \sum_{\mathbf{e} \in F} \mathbf{e} \in S \}$$ were introduced in \cite{Campillo} and used in \cite{pison03} to describe the combinatorics of $\mathcal{P}$. Specifically, the following result is verified (independently is $S$ is simplicial or not):

\begin{proposition}\label{Prop2.1PP} For every $\mathbf{a} \in S$ and $i = \{0 \ldots, t\},$ $$\widetilde{H}_i(T_\mathbf{a}) \cong W_i(\mathbf{a}),$$ where $\widetilde{H}_i(-)$ denotes the $i-$th reduced homology $\Bbbk$-vector space of $T_\mathbf{a}$.
\end{proposition}

\begin{proof}
See \cite[Propositon 2.1]{pison03}.
\end{proof}

Given $\mathbf{a} \in S$, we define $$C_\mathbf{a} = \Big\{ \mathbf{Y}^\mathbf{v} \in \Bbbk[\mathbf{Y}]\ \mid\ \deg_\mathcal{A}\big((\mathbf{v},\mathbf{u})\big) = \mathbf{a},\ \mbox{for some}\ \mathbf{u} \in \mathcal{Q} \Big\}.$$ Let $\Gamma_\mathbf{a}$ be the abstract simplicial complex with vertex set $C_\mathbf{a}$ defined as follows $$\Gamma_\mathbf{a} = \big\{ F \subseteq C_\mathbf{a} \mid \gcd(F) \neq 1 \big\}.$$

\begin{theorem}\label{Th 2}
For every $\mathbf{a} \in S$ and $i = \{0 \ldots, t\},$ $$\widetilde{H}_{i}(\Gamma_\mathbf{a})\cong\widetilde{H}_{i}(T_\mathbf{a}).$$
\end{theorem}

\begin{proof}
For each $\mathbf{Y}^\mathbf{v} \in C_\mathbf{a}$ define the simplicial complex $K_\mathbf{v} = \mathcal{P}\big(\mathrm{supp}(\mathbf{Y}^\mathbf{v})\big)$ to be the full subcomplex of $T_\mathbf{a}$ whose vertex set is $\mathrm{supp}(\mathbf{Y}^\mathbf{v})$. Now, $F \in T_\mathbf{a}$ if and only if there exists $\mathbf{Y}^\mathbf{v} \in C_\mathbf{a}$ with $\mathrm{supp}(\mathbf{Y}^\mathbf{v}) \supseteq F$. Indeed, if $F \in T_\mathbf{a}$, then $\mathbf{a} - \sum_{\mathbf{e} \in F} \mathbf{e} \in S$, that is to say, there exists $\mathbf{Y}^{\mathbf{v}}$ such that $\deg_\mathcal{A}\big((\mathbf{v},\mathbf{u})\big) = \mathbf{a}$, for some $\mathbf{u} \in \mathbb{N}^s$. Now, by taking the remainder of $\mathbf{Z}^{\mathbf{u}}$ upon division by $\mathcal{G}_\prec(I_\mathcal{A})$, we obtain a monomial $\mathbf{Y}^\mathbf{w} \mathbf{Z}^{\mathbf{u}'}$ with $\mathbf{u}' \in \mathcal{Q}$. So, $\mathbf{Y}^{\mathbf{v} + \mathbf{w}} \in C_\mathbf{a}$ and we are done; clearly, the opposite inclusion is true. Therefore, $\mathcal{K}^\mathbf{a} := \{K_\mathbf{v}\ \mid \mathbf{v} \in C_\mathbf{a}\}$ is a cover of $T_\mathbf{a}$.
Moreover, since $\cap_{i=1}^q K_{\mathbf{v}_i} \neq \varnothing$ if and only if $\gcd\big(\mathbf{Y}^{\mathbf{v}_1}, \ldots, \mathbf{Y}^{\mathbf{v}_q}\big) \neq 1$, we have that $\Gamma_\mathbf{a}$ is the nerve of $\mathcal{K}^\mathbf{a}$. Finally, since each non-empty finite intersection, $\cap_{i=1}^q K_{\mathbf{v}_i}$, is a full simplex, they are acyclic. Thus, by the Nerve Lemma (see \cite[Lemma 5.36]{MS}), we conclude that $\widetilde{H}_{i}(\Gamma_\mathbf{a})\cong\widetilde{H}_{i}(T_\mathbf{a}).$
\end{proof}

From the proof of Theorem \ref{Th 2}, it follows that if $\Gamma_\mathbf{a}$ is disconnected, then we may choose $\mathbf{Y}^{\mathbf{v}}, \mathbf{Y}^{\mathbf{v}'} \in C_\mathbf{a}$ in different connected components of $\Gamma_\mathbf{a}$ such that $\mathbf{Y}^{\mathbf{v}} \mathbf{Z}^{\mathbf{u}} - \mathbf{Y}^{\mathbf{v}'} \mathbf{Z}^{\mathbf{u}'} \in I_\mathcal{A},$ for some $\mathbf{u}$ and $\mathbf{u}' \in \mathcal{Q}$. Now, with the same notation as in Section 2, let $\mathbf{f} \in \Bbbk[\mathbf{Y}]^{\beta_0}$ whose $\sigma^{-1}(\mathbf{u})-$th and $\sigma^{-1}(\mathbf{u}')-$th coordinates are $\mathbf{Y}^\mathbf{v}$ and $-\mathbf{Y}^{\mathbf{v}'}$, respectively, and zeros elsewhere. Notice that the case $\mathbf{u} = \mathbf{u}'$ is not avoided, in this case, by construction, the only nonzero entry of $\mathbf{f}$ is minimal generator of $I_{S_E}$ in position $\sigma^{-1}(\mathbf{u})$. Thus, putting this together with the construction of the presentation of $\Bbbk[S]$ as $\Bbbk[\mathbf{Y}]-$module given in Section 2, we obtain that the isomorphisms $\widetilde{H}_{0}(\Gamma_\mathbf{a})\cong W_0(\mathbf{a})$ are explicitly described, for every $\mathbf{a} \in S$.

Finally, we give the explicit relation between the Betti numbers of the $\mathcal{A}-$graded minimal free resolution of $\Bbbk[S]$ and the Pison's resolution of $\Bbbk[S]$. Recall that $\beta_{i,\mathbf{a}}(I_\mathcal{A}) = \beta_{i+1,\mathbf{a}}(\Bbbk[S])$, for every $i \geq 0$.

\begin{corollary}
If $\bar \beta_{i,\mathbf{a}}(I_\mathcal{A})$ and $\beta_{i,\mathbf{a}}(\Bbbk[S])$ denote the $i-$th Betti number of $I_\mathcal{A} \subseteq \Bbbk[\mathbf Y, \mathbf Z]$ and the $i-$th Betti number of $\Bbbk[S]$ as $\Bbbk[\mathbf Y]-$module both in degree $\mathbf{a}$, respectively, then $$\bar \beta_{i,\mathbf{a}}(I_\mathcal{A}) = 0 \Longrightarrow \beta_{{i-\# F}, \mathbf{a} - \sum_{j \in F \subseteq B} \mathbf{b}_j}(\Bbbk[S]) = 0,$$ for every $F \subseteq B$ with $\# F \leq i+1$.
\end{corollary}

\begin{proof}
Let $D(l) = \big\{ \mathbf{a}' \in S \mid \dim \widetilde{H}_{l}(T_\mathbf{a}) = \beta_{l,\mathbf{a}'}(\Bbbk[S]) \neq 0 \big\},$ for each $l \geq 0$ and $C_i = \big\{ \mathbf{a} \in S \mid \mathbf{a}-\sum_{j \in F} \mathbf{b}_j \in D(i-\# F),\ \text{for some}\ F \subseteq B\ \text{with}\ \#F \leq i+1 \big\}$. Now, by \cite[Proposition 3.3]{Campillo} and by \cite[Theorem 9.2]{MS}, if $\bar\beta_{i,\mathbf{a}}(I_\mathcal{A}) = 0$, then $\mathbf{a} \not\in C_i$, for any $i \geq 0$, and our claim follows.
\end{proof}

\medskip\noindent
\textbf{Acknowledgments:} We thank Marcel Morales and Antonio Campillo for helpful comments and suggestions.



\begin{thebibliography}{}

\bibitem{BCMP}
\textsc{Briales, E.; Campillo, A.; Mariju\'an, C.; Pis\'on, P.} {Minimal Systems of Generetors for Ideals of Semigroups}. J. Pure Appl. Algebra, \textbf{124} (1998), 7--30.

\bibitem{BH}
\textsc{Bruns, W.; Herzog, J.}
\emph{Cohen-Macaulay rings}.
Cambridge studies in advanced mathematics, vol. 39, Cambridge University Press, 1993.

\bibitem{Campillo}
\textsc{A. Campillo, P. Gimenez.}
\emph{Syzygies of affine toric varieties}.
J. Algebra \textbf{225} (2000), 142--161.

\bibitem{CBPV}
\textsc{Briales, E.; Campillo, A.; Pison, P.; Vigneron, A.}
\emph{Simplicial Complexes and Syzygies of Lattice Ideals}.
Symbolic computation: solving equations in algebra, geometry, and engineering (South Hadley, MA, 2000), 169--183, Contemp. Math., \textbf{286}, Amer. Math. Soc., Providence, RI, 2001.

\bibitem{Zarzuela}
\textsc{Cortadellas Ben\'{\i}tez, T.; Jafari, R.; Zarzuela Armengou, S.}
\emph{On the Apery sets of monomial curves}.
Semigroup Forum \textbf{86} (2013), no. 2, 289--320.

\bibitem{numericalsgps} 
\textsc{Delgado, M.; Garc{\'i}a-S\'{a}nchez {P.A.}; Morais, J.},
\emph{``numericalsgps'': a {\sf {g}{a}{p}} package on numerical semigroups},
(\url{http://www.gap-system.org/Packages/numericalsgps.html}).

\bibitem{DGPS}
\textsc{Decker, W.; Greuel, G.-M.; Pfister, G.; Sch\"{o}nemann, H.}
\emph{Singular {4-0-2} --- {A} computer algebra system for polynomial computations}.
\url{http://www.singular.uni-kl.de} (2015).

\bibitem{einstein}
\textsc{Einstein, D.; Lichtblau, D.; Strzebonski, A.; Wagon, S.}
\emph{Frobenius numbers by lattice point enumeration}.
Integers \textbf{7} (2007), A15, 63 pp.

\bibitem{GSR}
\textsc{Garc\'{\i}a-S\'anchez, P.A.; Rosales, {J.C.}}
\emph{Finitely generated commutative monoids}.
Nova Science Publishers, Inc., Commack, NY, 1999.

\bibitem{GAP}
The GAP~Group, GAP -- Groups, Algorithms, and Programming, Version 4.7.5;
2014, (\url{http://www.gap-system.org}).

\bibitem{MOT}
\textsc{Marquez-Campos, G.; Ojeda, I.; Tornero, {J.M.}}
\emph{On the computation of the Apery set of numerical monoids and affine semigroups.}
Semigroup Forum. Published online: 01 October 2014

\bibitem{MS}
\textsc{Miller E.; Sturmfels, B.} \emph{Combinatorial Commutative Algebra}. Vol. 227 of Graduate Texts in Mathematics. Springer, New York. 2005.

\bibitem{Morales}
\textsc{Morales, M.; Dung, {N.T.}}
\emph{Gr\"{o}bner basis. a "pseudo-polynomial" algorithm for computing the Frobenius number.}
\url{arXiv:1510.01973 [math.AC]}

\bibitem{OjVi}
\textsc{Ojeda, I.; Vigneron-Tenorio, A.}
\emph{Simplicial complexes and minimal free resolution of monomial algebras.}
Journal of Pure and Applied Algebra, \textbf{214} (2010), 850--861.

\bibitem{pison03}
\textsc{{Pis\'on~Casares}, P.} \emph{The short resolution of a lattice ideal.} Proc. Amer. Math. Soc. \textbf{131}, 4, (2003), 1081--1091.

\bibitem{roune}
\textsc{Roune, B.H.}
\emph{Solving thousand-digit Frobenius problems using Gr\"{o}bner bases}.
J. Symbolic Comput. \textbf{43} (2008), no. 1, 1--7.


\end{thebibliography}
\end{document}